\documentclass[12pt]{article}

\usepackage[utf8]{inputenc}
\usepackage{geometry}
\usepackage{graphicx}
\usepackage{amsmath, amsthm, amssymb} 
\usepackage[ruled,vlined]{algorithm2e}
\usepackage[dvipsnames]{xcolor}
\usepackage{hyperref}
\usepackage{listings}
\hypersetup{
    colorlinks=true,
    linkcolor=black,
    filecolor=black,      
    urlcolor=black,
    citecolor=black
}

\theoremstyle{definition}

\newtheorem{theorem}{Theorem}

\newtheorem{definition}{Definition}
\newtheorem{lemma}{Lemma}
\newtheorem{corollary}{Corollary}

\newtheorem{example}{Example}
\newtheorem{remark}{Remark}

\newcommand{\R}{\mathbb{R}}

\begin{document}

\begin{center}
    \begin{Large} Computing Euclidean distance and maximum likelihood retraction maps for constrained optimization\end{Large}\\
    Alexander Heaton and Matthias Himmelmann\\
    \today
\end{center}

\begin{abstract}
    Riemannian optimization uses local methods to solve optimization problems whose constraint set is a smooth manifold. A linear step along some descent direction usually leaves the constraints, and hence \textit{retraction maps} are used to approximate the exponential map and return to the manifold. For many common matrix manifolds, retraction maps are available, with more or less explicit formulas. For implicitly-defined manifolds, suitable retraction maps are difficult to compute. We therefore develop an algorithm which uses homotopy continuation to compute the Euclidean distance retraction for any implicitly-defined submanifold of $\R^n$, and prove convergence results. 
    
    We also consider statistical models as Riemannian submanifolds of the probability simplex with the Fisher metric. Replacing Euclidean distance with maximum likelihood results in a map which we prove is a retraction. In fact, we prove the retraction is second-order; with the Levi-Civita connection associated to the Fisher metric, it approximates geodesics to second-order accuracy.
\end{abstract}

\section{Introduction}\label{introduction}

We consider constrained optimization in the case where the constraint set is an implicitly-defined manifold or real algebraic variety $\mathcal{M} \subset \R^n$. We also consider statistical models $\mathcal{M} \subset \Delta_{n-1} \subset \R^n$ in Section \ref{section:statistical-models}. The algorithms we develop are local, motivated by examples which are too large for global methods like Lagrange multipliers. In Riemannian optimization, the basic idea of local methods is simple. To find the argmin, start somewhere on the manifold and move along a \textit{descent direction} in the tangent space. This (usually) takes you off the manifold, so you must \textit{retract} (see Definition \ref{def:retraction}) back to the manifold, and repeat. Eventually you hope to find $p \in \mathcal{M}$ locally minimizing $f$. Robust techniques and theoretical results accompany a well-chosen Riemannian metric and retraction map \cite{matrixmaifolds, boumal2020intromanifolds, Ring2012}. However, such methods require prior knowledge of a suitable retraction map, ideally with an explicit formula. These are available for many common matrix manifolds, making Riemannian optimization a powerful tool in applications. However, for less common, implicitly-defined manifolds or algebraic varieties, there are no readily available formulas for retractions. Therefore, retraction maps applicable to an arbitrary manifold are desirable.

To construct retraction maps on embedded manifolds or algebraic varieties, intuition suggests retracting to \textit{the closest point} on the manifold. Such maps are called \textit{metric projections} in \cite[p.111]{boumal2020intromanifolds}. 
When we step linearly off a sphere, we can normalize the resulting point. When we leave the manifold of fixed rank matrices, we can apply the singular value decomposition. However, for most implicitly-defined manifolds, retracting to the closest point does not admit an explicit formula or easy solution, and ``it is difficult to compute in general'' \cite[p.164]{boumal2020intromanifolds}. However, if we can compute metric projections, it has been shown \cite{AbsMal2012} that they are second-order retractions \cite[Def. 5.41]{boumal2020intromanifolds}, a desirable property since they match geodesics on the manifold up to second-order. After all, retractions should approximate the exponential map, which is useful provided that retraction is easier to compute in practice.

On the one hand this is an obvious choice; retract to the closest point on the manifold. Upon further reflection an absurdity is noted: 
We begin with a constrained optimization problem on $\mathcal{M}$, to minimize some function $f$. Then we suggest solving another constrained optimization problem on $\mathcal{M}$ repeatedly, at every local step, with the goal of finding the closest point. In particular, we do not have an explicit formula. Have we replaced one hard problem with many hard problems? The absurdity is resolved by use of a small trick.
As we explain in Section \ref{section:euclidean-distance-retraction}, we can use homotopy continuation methods \cite{Bertini, HC.jl, sommesewampler} to easily find the closest point at each local step. We do this by creating a system of equations whose solutions we know, and deforming them via a homotopy. To be clear, without this tool, the metric retraction might be an absurd choice, since we cannot compute it easily. Note that, in contrast to \cite{ContinuationMethodsForRiemannianOptim}, we use homotopy continuation locally, rather than globally. In other words, we only use homotopy continuation to compute the retraction map at each local step.

In Section \ref{section:euclidean-distance-retraction} we describe Algorithm \ref{alg:euclidean-distance-retraction}, which uses homotopy continuation to compute the Euclidean distance retraction (see Theorem \ref{thm:MAIN:euclidean-distance-is-a-retraction}) for any implicitly-defined submanifold of $\R^n$ and we provide convergence guarantees in Theorem \ref{thm:convergence-result} and Corollary \ref{cor:eulersmethodconverges}. In Section \ref{section:statistical-models}, we view statistical models as Riemannian submanifolds of the probability simplex equipped with the Fisher metric. In this setting, we prove Theorem \ref{thm:fisher-retraction}, which shows that ideas from maximum-likelihood estimation may be used to create a second-order retraction applicable to arbitrary statistical models at smooth points. To show second-order, we use geodesics and covariant derivatives with the Levi-Civita connection on $\mathcal{M}$ corresponding to the Fisher metric. Theorem \ref{thm:fisher-retraction} shows that replacing the linear step $p + tv$ from Algorithm \ref{alg:euclidean-distance-retraction} by the quadratic curve $\left( p_i + t v_i + t^2 \frac{v_i^2}{4p_i} \right)_{i=1}^n$ in the usual coordinates on $\R^n$ allows the resulting retraction $R_p(tv)$ to follow geodesics on $\mathcal{M} \subset \Delta_{n-1} \subset \R^n_{>0}$ to second-order accuracy.  Section \ref{section:usage} describes an easy to use Julia package which implements Algorithm \ref{alg:euclidean-distance-retraction} and briefly demonstrates its usage, while Section \ref{section:experiments} demonstrates the algorithm's perfomance.

\textbf{Acknowledgements:} We thank the Fields Institute Thematic Program on Geometric Constraint Systems, Framework Rigidity, and Distance Geometry, which supported the first author as a Fields Postdoctoral Fellow. We thank Myfanwy Evans for suggesting an applied optimization problem during the Fields Program which was both interesting and difficult to solve. This work arose from our attempts to solve that problem.  We also thank Paul Breiding, who suggested using the Euclidean distance equations for the homotopy, rather than another homotopy we had used initially to solve the problem.

\section{Euclidean Distance Retraction by Homotopy}\label{section:euclidean-distance-retraction}

We begin by considering submanifolds $\mathcal{M} \subset \R^n$ with the usual Euclidean metric, and which are implicitly-defined, meaning a smooth function $g:\R^n \to \R^m$ is known such that $\mathcal{M} = g^{-1}(0)$. Since our methods are local, we also consider real algebraic varieties $\mathcal{M} \subset \R^n$, which are smooth manifolds outside the singular locus. Given any smooth function $f:\R^n \to \R$ our main objective is to solve the constrained optimization problem of minimizing $f$ restricted to $\mathcal{M}$.

Given a point $p \in \mathcal{M}$, we can evaluate the Euclidean gradient of $f$ at $p$, obtaining a vector $\nabla f|_p \in \R^n$. Projecting $-\nabla f|_p$ onto the tangent space $T_p \mathcal{M}$, we obtain a \textit{descent direction} $v \in T_p \mathcal{M}$. There are (better) ways to produce descent directions as well, for example \cite[Alg. 6.3 (Riemannian trust-region)]{boumal2020intromanifolds} or \cite[Alg. 5 (Riemannian Newton method)]{matrixmaifolds}. After obtaining a descent direction, the idea is to move from $p \in \mathcal{M}$ to $p+v \notin \mathcal{M}$, and then \textit{retract} back to the manifold in such a way as to approximate the exponential map. This leads us to the following

\begin{definition}\label{def:retraction}
Let $\mathcal{M} \subset \R^n$ be a manifold and $p \in \mathcal{M}$ a point. A \emph{(first-order) retraction} at $p$ is a smooth map $R_p:T_p \mathcal{M} \to \mathcal{M}$ satisfying
\begin{enumerate}
    \item $R_p(0) = p$, and
    \item $d(R_p)(0):T_p \mathcal{M} \to T_p \mathcal{M}$ is the identity map.
\end{enumerate}
Equivalently, for each curve $c:(-\epsilon,\epsilon) \subset \R \to \mathcal{M}$ with $c(t) = R_p(tv)$ for some $v \in T_p \mathcal{M}$ we have $c'(0) = v$ and $c(0)=p$. A \emph{local retraction at p} is a retraction defined in some neighborhood of $0 \in T_p \mathcal{M}$, but satisfying the same two properties above.

A retraction is called \textit{second-order} if, in addition, for all $v \in T_p \mathcal{M}$
\begin{equation*}
    \frac{d^2}{dt^2} R_p(tv) |_{t=0} \in N_p \mathcal{M}.
\end{equation*}
This implies that the curve $R_p(tv)$ matches geodesics on $\mathcal{M}$ up to second order at $t=0$, i.e. $R_p(tv) = \text{geodesic}(t;p,v) + O(t^3)$ as $t\to 0$ \cite[Prop. 3]{AbsMal2012}.
\end{definition}

We note that a retraction map has a different definition in topology, although they are related. Our Definition \ref{def:retraction} is appropriate in the context of \textit{Riemannian optimization}, and matches the definition in the textbook \cite{boumal2020intromanifolds}, for instance. Retractions were first defined in \cite{Adler-newtonsmethod-retraction-definition}.
    In the more general setting of Section \ref{section:statistical-models} we consider Riemannian submanifolds of a Riemannian manifold which is not Euclidean space. In that case, the condition for a second-order retraction must instead be that the covariant derivative of the vector field defined by the velocity of the curve $R_p(tv)$ lies in the normal space at $t=0$ \cite[Def. 8.64 and Eqn. 8.27]{boumal2020intromanifolds}. We delay the additional details until Section \ref{section:statistical-models}.

By a result in \cite{AbsMal2012}, finding the closest point with respect to Euclidean distance defines a local retraction in the sense of Definition \ref{def:retraction} on any submanifold of Euclidean space.

\begin{theorem}[\cite{AbsMal2012}]\label{thm:MAIN:euclidean-distance-is-a-retraction}
Let $\mathcal{M} \subset \R^n$ be a smooth manifold or real algebraic variety. For any smooth point $p \in \mathcal{M}$, define the relation $R_p \subset T_p \mathcal{M} \times \mathcal{M}$ by
\begin{equation*}
    R_p = \{ (v, u) \in \R^n \times \R^n \,\, : \,\, u \in \text{argmin}_{y \in \mathcal{M}} |p + v - y| \}.
\end{equation*}
There exists a neighborhood $U$ of $0$ in $T_p \mathcal{M}$ such that $R_p$ defines a local, second-order retraction. The curve $t \mapsto R_p(tv)$ matches the geodesic for $v$ up to second-order at $p$.
\end{theorem}

\begin{example}
\label{example:sphereretraction}
Denote the unit sphere in $\R^n$ by $S^{n-1}$. The expression
\begin{equation*}
    R_x(v) = \frac{x+v}{|x+v|}
\end{equation*}
provides a smooth map $R_x:TS^{n-1}\rightarrow S^{n-1}$. As can be easily checked, this formula defines a retraction, called the Euclidean distance retraction, which admits an explicit formula in this simple example. Consequently, we could rewrite
\begin{equation*}
    R_x(v) = \arg\min_{y \in S^{n-1}} |y-(x+v)|,
\end{equation*}
since the argmin is unique in every point except $0$, which never equals $x+v$.
\end{example}

\begin{example}
\label{example:orthogonalgroup}
Choose $\mathcal{M}=O_n$ the orthogonal group. Denote by $qr: GL_n(\R)\rightarrow O_n \times S^+_{upp}(n)$ the QR decomposition of a real, invertible $n\times n$ matrix into an orthogonal matrix and an upper triangular matrix with strictly positive diagonal entries. Notice that GL$_n(\mathbb{R})$ is the complement of a hypersurface in $\mathbb{R}^{n\times n}$, cut out by $\{A\in \mathbb{R}^{n\times n}~|~\det (A)\neq 0\}$. Then, the map 
\[R_X(\xi) = \pi_1(qr(X+\xi)) \]
is a (local) retraction on the orthogonal group's tangent bundle $TO_n$ with $\pi_1$ the projection onto its first argument. A proof can be found in \cite[Ex. 4.1.2]{matrixmaifolds}
\end{example}

In general, the Euclidean distance retraction will not admit an explicit formula, except in special cases like those above. Therefore, it is desirable to develop techniques for computing the Euclidean distance retraction on arbitrary, implicitly-defined manifolds, in case an explicit formula is not available. This will be accomplished via Algorithm \ref{alg:euclidean-distance-retraction} below.

\subsection{Homotopy continuation for the Euclidean distance retraction}\label{section:homotopy-continuation}

In order to compute $R_p(v)$ for the Euclidean distance retraction of Theorem \ref{thm:MAIN:euclidean-distance-is-a-retraction}, we need to solve an optimization problem over the constraint set $\mathcal{M}$, namely, minimizing Euclidean distance. At first glance, this is absurd. We started with an optimization problem (minimizing some other function $f:\R^n \to \R$ restricted to $\mathcal{M}$), and now we suggest taking local steps by solving more optimization problems with the same constraints. However, we have tricks to solve these local optimization problems efficiently, so our proposal is surprisingly useful. In this section we describe how \textit{homotopy continuation} allows us to easily compute the Euclidean distance retraction. 

Given any smooth function $f:\R^n \to \R$ we consider the constrained optimization problem of minimizing $f$ restricted to $\mathcal{M} = g^{-1}(0)$, where $g:\R^n \to \R^m$ is a smooth map and $n-m = \text{dim} (\mathcal{M})$. In what follows, let $dg_x$ denote the $m\times n$ Jacobian matrix of $g$ in the standard basis, evaluated at the point $x$, and let $dg_x^T$ denote the $n \times m$ transpose matrix. Given a point $p \in \mathcal{M}$, and a descent direction $v \in T_p \mathcal{M}$, the idea is to move from $p \in \mathcal{M}$ to $p+v$, which in general will not lie on the manifold $\mathcal{M}$. Subsequently, we want to \textit{retract} back to the manifold in such a way as to approximate the exponential map.
Theorem \ref{thm:MAIN:euclidean-distance-is-a-retraction} shows that computing the Euclidean closest point $ R_p(v) \in \mathcal{M}$ to $p + v \in \R^n$ achieves this goal.

\textbf{Newton's Method.} As suggested in \cite[Ex. 4]{Adler-newtonsmethod-retraction-definition}, one method to find $R_p(v)$ is to apply NM (Newton's method) to the Lagrange multiplier system. We will call this \textit{Algorithm 0}. Set $u = p+v $. Using the variables $(x,\lambda)\in \R^n\times \R^m$ define $G:\R^n \times \R^m \to \R^m \times \R^n$ by 
\begin{equation*}
    G(x,\lambda;u) = (g(x), \, dg_x^T \lambda - u + x).
\end{equation*}
We view $u$ as a parameter, and $(x,\lambda)$ as variables. Then any solution $(x_u,\lambda_u)$ of $G(x,\lambda;u) = 0$ for fixed $u = p+v$ encodes a point $x_u \in \mathcal{M}$ which is first-order critical for minimizing Euclidean distance from $u$ to $\mathcal{M}$.

By applying NM to $G(x,\lambda;u)=0$ one hopes to find the solution $x_u = R_p(v)$. A problem is that NM may diverge, or converge slowly. Consider for example the twisted cubic $\mathcal{C}$, parametrized by $t\mapsto (t,t^2,t^3)$. Figure \ref{fig:Newtonsnormal} shows part of the normal space at $(1,1,1)$ with points where Newton's method takes more than $10$ iterations to converge displayed in red. 

\begin{figure}[h!]
    \centering
    \includegraphics[width=0.75\linewidth]{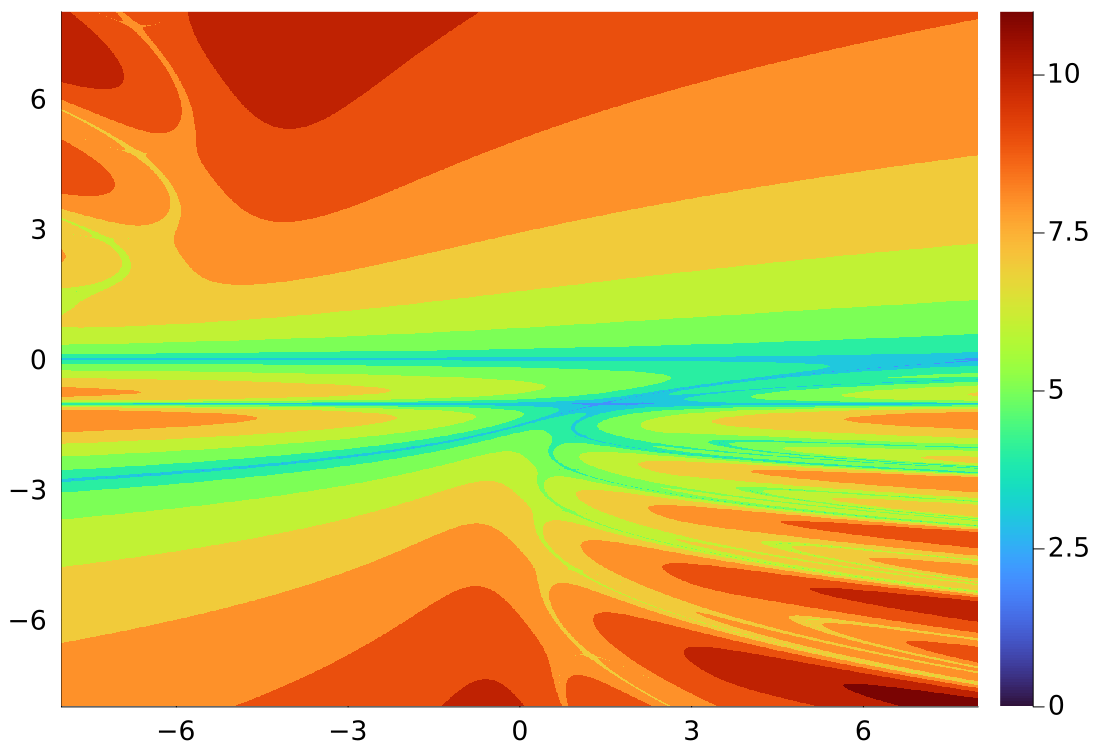}
    \caption{\small Representation of the twisted cubic's normal space at $p=(1,1,1)$ with the picture's origin corresponding to $p$. The colors show how many iterations Newton's method takes to converge.} 
    \label{fig:Newtonsnormal}
\end{figure}

Even if NM converges, the solution $x_u$ may not be $R_p(v)$, but instead some other critical point. NM is not guaranteed to converge to the closest solution to the starting point. These problems disappear if we find an initial guess $(x_0,\lambda_0)$ close enough to the desired solution $G(R_p(v),\lambda_{u}; u=p+v)=0$. Finding such an initial guess will be our goal.

\textbf{Homotopy Continuation}.
Given a system of equations $f(x)=0$ whose solutions are unknown, homotopy continuation seeks to deform the known solutions of another system $g(x)=0$ into the previously unknown solutions of $f(x)=0$. One usually embeds both systems into a parametrized family of systems of equations $h(x;u) = 0$ such that $g(x) = h(x;u_0)$ while $f(x)=h(x;u_1)$. Choose a path $u(t)$ between $u_0$ and $u_1$. Then $\frac{d}{dt}h(x(t); u(t)) = 0$ is a system of ordinary differential equations whose solution $x(t)$ encodes a path between the known solution $g(x(0))=h(x(0);u(0))=0$ and the unknown solution $f(x(1)) = h(x(1);u(1))=0$. These odes are easy to solve numerically since at any $t$ we know the solution $x(t)$ satisfies the system of equations $h(x; u(t))=0$. Therefore, we may apply NM (Newton's method) to correct any errors arising during numerical path tracking. This procedure is called homotopy continuation. Already, there exists efficient software specially designed for solving the ODEs arising in homotopy continuation, see for example \cite{Bertini, HC.jl, Hom4PSArticle, nag-macaulay2, PHCpack}.

For our problem, we view the Lagrange multiplier system $G(x,\lambda; u)=0$ as parametrized by $u$. Taking $u = p$ gives us a system with known solution $(x,\lambda) = (p,0)$. Taking $u=p+v$ gives a system whose solutions include $(x = R_p(v), \lambda^*)$ for some $\lambda^*$ we ignore. One issue that may occur while solving the homotopy ODEs is that for some parameter value $u(t)$ the Jacobian of the system of equations becomes singular. Numerical path tracking may drastically fail if this occurs. For that reason, one usually chooses a path $u(t)$ through complexified parameter space. Since the parameter values with singular Jacobian form a set of complex codimension one, and real codimension two, a random path $u(t)$ in complexified space will avoid the bad parameter values with probability one (see Lemma 7.1.2 of \cite{sommesewampler}). The price paid is that a real-valued solution $x(0)$ may become a complex-valued solution $x(1)$ by the end of the homotopy. See Figure \ref{fig:whenHcGoesWrong} for an example. Since we are only interested in one particular real solution $x=R_p(v)$, this is a problem. 
\begin{figure}[h!]
    \centering
    \includegraphics[width=0.73\linewidth]{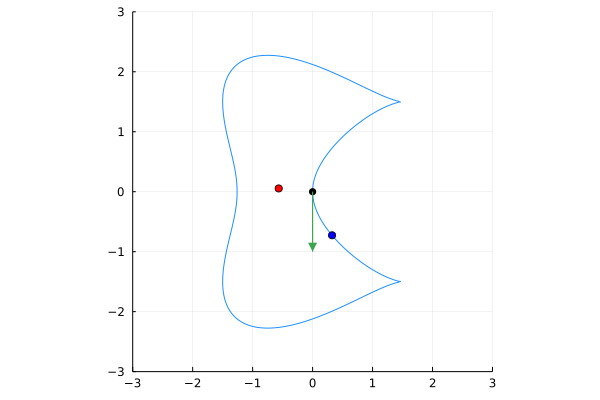}
    \caption{\small Bicuspid curve $\mathcal{C}_a =  \{(x,y)~:~(x^2-a^2)\cdot(x-a)^2+(y^2-a^2)^2=0\}$ for $a=1.5$ with initial point $p=(0,0)$ and tangent direction $v=(0,-1)$. Using a homotopy in complex parameter space computes a complex-valued critical point whose real-part is shown in red, instead of the intended critical point $R_p(v)$ shown in blue.}
    \label{fig:whenHcGoesWrong}
\end{figure}

The solution is given by Lemma \ref{lemma:arbitrary-codimension-submanifolds}, which we prove below, and which is the main theoretical result underlying our method. The Lemma says that if we choose a real-valued parameter path $u(t)$ that stays within a certain distance from $\mathcal{M}$, the relevant Jacobian is always invertible. This Lemma only applies to the Lagrange multiplier system of equations specific to our problem. Therefore, we may confidently apply homotopy continuation using the real numbers, and by the tubular neighborhood theorem we are guaranteed to find the real-valued $R_p(v)$ as our final solution, rather than some other unwanted critical point. All of this works provided our path $u(t)$ between $p$ and $p+v$ stays close enough to $\mathcal{M}$.

\textbf{A simplified version of homotopy continuation.} In order to prove some convergence results, we will first analyze a simplified version of homotopy continuation. Consider a curve $u(t)$ with $t\in [0,1]$, $u(0)=p$ and $u(1)=p+v$. Instead of applying Newton's method only to $G(x,\lambda;u(1))=0$, we partition the interval $[0,1]$ as $0=t_0<t_1<\dots < t_s = 1$ and apply NM to each system $G(x,\lambda;u(t_j))=0$, using the previous solution as the initial guess for the next. This approach is summarized in pseudocode below.

\vspace*{3mm}
\begin{algorithm}[H]
\SetAlgoLined
\SetKwInOut{Input}{Inputs}
\SetKwInOut{Output}{Output}
\DontPrintSemicolon
\Input{A point $p \in \mathcal{M}$, vector $v \in T_p \mathcal{M}$, a path $u(t)$ with $u(0)=p$ and $u(1)=p+v$, and a partition $0=t_0 < t_1 < \cdots < t_s=1$.}
\Output{The Euclidean distance retraction $R_p(v)$.}
\;
Set $x(t_0) = p$ and set $\lambda(t_0) = 0$.\;
\For{$j \in \{1,2,\dots,s\}$}{Apply Newton's method to the system $G(x,\lambda; u(t_j))=0$ with initial point $(x(t_{j-1}),\lambda(t_{j-1}))$, obtaining a solution $(x^*,\lambda^*)$. Set $x(t_j) = x^*$ and $\lambda(t_j) = \lambda^*$.}
 \caption{A simplified Euclidean distance retraction}
 \label{alg:simplified-eucl-dist-retraction}
\end{algorithm}
\vspace*{3mm}

At first glance, Algorithm \ref{alg:simplified-eucl-dist-retraction} seems to worsen the performance of Algorithm $0$. But, if we can ensure that each intermediate solution lies in the next solution's radius of quadratic convergence, this approach may be more efficient and guarantee convergence. This will be made more precise in Theorem \ref{thm:convergence-result} and Corollary \ref{cor:eulersmethodconverges}. Still, it should be stressed that our results use special properties of the Lagrange multiplier system, parametrized by $u$, for finding the Euclidean closest point to $u$ on $\mathcal{M}$. First we prove a technical lemma on which our other results depend.

\begin{lemma}\label{lemma:arbitrary-codimension-submanifolds}
Let $\mathcal{M} = g^{-1}(0)$ for some $C^2$ map $g:\R^n \to \R^m$ and let $u^* \in \R^n$ be given. Let $x^*$ be the closest point on $\mathcal{M}$ to $u^*$ and assume that $dg(x^*)$ has full rank. If
\begin{equation}\label{eqn:inequality-hypersurfaces}
    |u^* - x^*| < \frac{1}{|S|}\sum_{i\in S} \frac{|\nabla g_i(x^*)|}{|d^2g_i(x^*)|} ~~~~~\text{ for } S = \{i\in\{1,\dots , m\} : |d^2g_i(x^*)| \neq 0\}
\end{equation}
then $\frac{\partial G}{\partial(x,\lambda)}(x^*,\lambda^*)$ is invertible. 

In the formula above $|u^* - x^* |$ and $|\nabla g_i(x^*)|$ are the usual norms on $\R^n$, $d^2g_i(x^*)$ is the Hessian of $g_i$ at the point $x^*$, and $|d^2g_i(x^*)|$ denotes the induced matrix norm. If $|d^2g_i(x^*)| = 0$ for all $i$, we interpret the right side of (\ref{eqn:inequality-hypersurfaces}) as $\infty$ and again claim that $\frac{\partial G}{\partial(x,\lambda)}(x^*,\lambda^*)$ is invertible.
\end{lemma}

\begin{proof}
Recall $G(x,\lambda;u) = (g(x), dg_x^T \lambda - u + x)$ and so $\partial G/\partial(x,\lambda)$ has a block structure
\begin{equation*}
    \begin{bmatrix}
        d g(x)^T & 0\\
        I_n + \sum_{i=1}^m\lambda_i d^2 g_i(x) & d g(x) 
    \end{bmatrix}.
\end{equation*}
To avoid a proliferation of stars we will sometimes omit the $^*$ in $x^*,\lambda^*,u^*$ in the calculations below.
Since $G(x,\lambda,u) = 0$ then $\lambda^T dg(x) = u - x$ and we obtain after taking norms and applying subadditivity:
\begin{equation}
    \label{eq:first-order-ux}
    |u-x| = |\sum_{i=1}^m \lambda_i \nabla g_i(x) | \leq \sum _{i=1}^m |\lambda_i \nabla g_i(x)| = \sum_{i=1}^m |\lambda_i| \,\, |\nabla g_i(x)|.
\end{equation}
Since $dg_i(x)$ has full rank, the matrix $\partial G/\partial(x,\lambda)$ is invertible if and only if $I_n + \sum_{i=1}^m\lambda_i d^2 g_i(x)$ is invertible. This is the case when all eigenvalues of the matrix $A=\sum_{i=1}^m\lambda_i d^2 g_i(x)$ are different from $-1$. The spectral radius is the maximum modulus of all eigenvalues. Also, the spectral radius is less than or equal to any matrix norm induced from a norm on the vector space. Therefore, assuming that $|A|<1$ implies that $\partial G/\partial(x,\lambda)$ is invertible. If $|dg^2_i(x)|=0$ for all $i$, it holds that $|A|=0$, implying that $\partial G/\partial(x,\lambda)$ is invertible.

Otherwise, since
\[|\sum_{i=1}^m\lambda_i d^2 g_i(x)|\leq \sum_{i\in S} |\lambda_i| \, |d^2g_i(x)| ~~~~~\text{ for } S=\{i\in\{1,\dots , m\} : |d^2g_i(x^*)| \neq 0\},\]
by assuming for each $i$ that $|\lambda_id^2g_i (x)|<\frac{1}{|S|}$ we can ensure that $|A|<1$. Equivalently, $|\lambda_i|<\frac{1}{|S|\cdot |d^2g_i(x)|}$ for all $i\in S$. Inserting this into inequality $(\ref{eq:first-order-ux})$, we see that if
\[|u-x| ~<~ \sum_{i\in S} \frac{|\nabla g_i(x)|}{|S| \,\, |d^2g_i(x)|} ~+~ \sum_{i\notin S} |\lambda_i \nabla g_i(x)|\]
then the matrix $\partial G/\partial(x,\lambda)$ is invertible. In particular, the theorem's assumption implies this inequality, proving the claim.
\end{proof}

In general, Newton's method (NM) may diverge or converge to a critical point other than $R_p(v)$, so Algorithm 1 need not succeed. However, for the problem of computing the retraction $R_p(v)$ we can now prove the following

\begin{theorem}\label{thm:convergence-result}
    Let $g:\R^n \to \R^m$ be three times continuously differentiable. Let $\mathcal{M} = g^{-1}(0) \subset \R^n$ be a smooth manifold with $dg_x$ full rank for every $x \in \mathcal{M}$, and let $u(t)$ be a smooth path in $\R^n$ staying within the tubular neighborhood of $\mathcal{M}$ for which there is a unique closest point. For each $t$, let $x(t)$ be the closest point to $u(t)$. If the path $u(t)$ satisfies
    \begin{equation}\label{eqn:hypersurface-inequality}
        \text{dist}(u(t), \, \mathcal{M}) < \frac{1}{|S|}\sum_{i\in S} \frac{|\nabla g_i(x(t))|}{|d^2g_i(x(t))|},
    \end{equation}
    where $S = \{ i : |d^2g_i(x(t))| \neq 0 \}$,
    then there exists a partition $0=t_0<t_1<\cdots<t_s=1$ such that each of the $s$ applications of Newton's method in Algorithm \ref{alg:euclidean-distance-retraction} converges, and hence Algorithm \ref{alg:euclidean-distance-retraction} converges, and outputs the correct critical point $R_p(v)$.
\end{theorem}

\begin{proof}
    For each $u(t)$, let $z(t)$ denote the pair $(x(t),\lambda(t))$ of the closest point on $\mathcal{M}$ to $u(t)$ and its associated Lagrange multiplier. We need to show that any partition of $[0,1]$ may be refined until $(x(t_j),\lambda(t_j))$ lies in the region of convergence of NM applied to $G(x,\lambda;u(t_{j+1}))=0$, and so may be used as starting point for the Newton iteration with guaranteed convergence. 
    
    Let $G'$ denote $\partial G/\partial(x,\lambda)$. By Lemma \ref{lemma:arbitrary-codimension-submanifolds}, $G'(z(t);u(t))$ is invertible and hence by the implicit function theorem we can solve locally for $z$ in terms of $u$. Thus the smooth path $u(t)$ gives rise to a smooth path $z(t)$ such that $G(z(t);u(t))=0$ for all $t$. The path $\{ (z(t),u(t)):t \in [0,1]\}$ is compact and every point has a neighborhood $U = B_{\delta(t)}(z(t)) \times B_{\delta(t)}(u(t))$ such that if $(z_0,u_0) \in U$ then $G'(z_0;u_0)^{-1}$ exists, since by Lemma \ref{lemma:arbitrary-codimension-submanifolds} we know $G'(z(t);u(t))^{-1}$ exists and since $G'(z;u)^{-1}$ is continuous in $z$ and $u$. By compactness there exists $r > 0$ such that for all $t$, if $z_0 \in \overline{ B_r(z(t)) } = K_t$ then $G'(z_0;u(t))^{-1}$ exists.

    In preparation for using Taylor's theorem to prove convergence of NM, define
    \begin{equation*}
        M(t) = \sup \frac{1}{2} \,\, |G'(w_0;u(t))^{-1}| \,\, |G''(w_1;u(t))|,
    \end{equation*}
    where the supremum is taken over all $(w_0,w_1) \in K_t \times K_t$. Then $M(t)$ is a continuous function of $t$ and attains its maximum value $M \geq 0$ over $t \in [0,1]$. By refining any initial partition, we may find a partition $0 = t_0 < t_1 < \cdots < t_s = 1$ such that
    \begin{equation*}
        |z(t_i) - z(t_{i+1})| < \min \{ r, 1/M \}.
    \end{equation*}
    If $M\equiv 0$, we interpret $1/M$ as infinity and the minimum becomes $r$. Let $z_u$ denote the closest point and Lagrange multiplier pair corresponding to $u$. By Taylor's theorem, for any $z_n$ we have
    \begin{equation}\label{eqn:Taylor-order-two}
        0 = G(z_u;u) = G(z_n;u) + G'(z_n;u)(z_u - z_n) + \frac{1}{2} G''(\Tilde{z};u)(z_u - z_n, z_u - z_n),
    \end{equation}
    for some $\Tilde{z}$ lying on the line segment between $z_n$ and $z_u$. Newton's method (NM) sets
    \begin{equation}\label{eqn:Newton-iteration}
        z_{n+1} = z_n - G'(z_n;u)^{-1} G(z_n;u).
    \end{equation}
    Use (\ref{eqn:Taylor-order-two}) to replace $G(z_n;u)$ in (\ref{eqn:Newton-iteration}), rearrange, and take norms to obtain
    \begin{align*}
        |z_{n+1} - z_u| &= \frac{1}{2} \, |G'(z_n;u)^{-1} G''(\Tilde{z};u)(z_u-z_n,z_u - z_n) |\\
        & \leq \frac{1}{2} \,\, |G'(z_n;u)^{-1}| \,\, |G''(\Tilde{z};u)| \,\, |z_u-z_n|^2.
    \end{align*}
    Let $e_n = |z_n - z_u|$. Then by induction we have $e_n \leq \frac{1}{M}(M e_0)^{2^n}$ so that if $e_0 < 1/M$ we have $e_n \to 0$ and NM converges. Since $|z(t_i) - z(t_{i+1})| < \min \{ r, 1/M \}$ we see that $z(t_i) \in \overline{B_r(z(t_{i+1}))}$, along with any point on the line segment between, so that using $z(t_i)$ as the initial point for NM applied to $G(z;u(t_{i+1}))=0$ implies that $|z(t_i) - z(t_{i+1})| = e_0 < 1/M$, which completes the proof. 
\end{proof}

\textbf{Predictor-corrector approach.} 
In practice, the procedure above may lead to a rather small discretization of the interval $[0,1]$. A better estimate to initialize NM 
is to use Euler's method \cite[p.22f.]{sommesewampler}, i.e. an approximation of $z_u$ informed by Taylor expansion. First, let us derive the formula for Euler's method to get an error estimate. Assume we are in a point $z=(x,\lambda)$ with parameter $u$. We want the subsequent value $z+\Delta z$ to also solve the polynomial system $G$ at the parameter $u+\Delta u$, so we assume that $G(z+\Delta z,~u+\Delta u)~=~0$.
To obtain the corresponding Taylor expansion, we write
\begin{eqnarray}
\label{eqn:GTaylorExpansion}
G(z+\Delta z,~u+\Delta u) ~=~ G(z,u) + \frac{\partial G}{\partial z}(z,u)\cdot \Delta z ~+~ \frac{\partial G}{\partial u}(z,u) \cdot \Delta u ~+~\mathcal{O}\left(\lvert \lvert \Delta z \lvert \lvert ^2\right).
\end{eqnarray}
With the additional assumption $G(z,u)\approx 0$, i.e. that $z$ is an approximate zero of the system $G$ at parameter $u$, and that $G'$ is invertible at $z$ (see Lemma \ref{lemma:arbitrary-codimension-submanifolds}), Equation (\ref{eqn:GTaylorExpansion}) yields the prediction
\begin{eqnarray}
\label{eqn:deltazexpression}
\Delta z ~ = ~ - \left(\frac{\partial G}{\partial z}(z,u)\right)^{-1} \frac{\partial G}{\partial u}(z,u)\cdot \Delta u.
\end{eqnarray}
Having derived a formula for Euler's method, we can now record the main idea of the corresponding predictor-corrector method in the following algorithm.

\vspace*{3mm}
\begin{algorithm}[H]
\SetAlgoLined
\SetKwInOut{Input}{Inputs}
\SetKwInOut{Output}{Output}
\DontPrintSemicolon
\Input{A point $p \in \mathcal{M}$, vector $v \in T_p \mathcal{M}$, a path $u(t)$ with $u(0)=p$ and $u(1)=p+v$, and a partition $0=t_0 < t_1 < \cdots < t_s=1$.}
\Output{The Euclidean distance retraction $R_p(v)$.}
\;
Set $x(t_0)=x(0)=p$ set $\lambda(t_0)=\lambda(0) = 0$. Let $z=(x,t)$ and $\Delta u_j = u_j-u_{j-1}$.\;
\For{$j \in \{0,1,\dots,s-1\}$}{
Predict $z=z_{j}+\Delta z$ by solving $\frac{\partial G}{\partial z,\lambda}(z_{j};u_{j})\Delta z = -\frac{\partial G}{\partial u}(z_{j};u_{j}) \Delta u_{j+1}$ (Euler).\;
Correct $z$ by solving $G(z,u_{j+1})=0$, yielding $z_{j+1}$ (Newton's method).\;
}
 \caption{Euclidean distance retraction by homotopy}
 \label{alg:euclidean-distance-retraction}
\end{algorithm}
\vspace*{3mm}

In Section \ref{section:usage} we present an implementation of Algorithm \ref{alg:euclidean-distance-retraction} in Julia and show how it performs favorably in comparison with Algorithm 0 on several examples. Nevertheless, the fact that the local truncation error of Euler's method is in $\mathcal{O}(||\Delta z||^2)$ and the error before the application of Algorithm \ref{alg:simplified-eucl-dist-retraction} is in $\mathcal{O}(||z_u-z||)$, already gives an intuition that in practice a predictor-corrector scheme involving Euler's method should perform better than just applying Newton's method. Using Theorem \ref{thm:convergence-result}, we can even prove this hypothesis theoretically.

\begin{corollary}\label{cor:eulersmethodconverges}
    Given the same assumptions as in Theorem \ref{thm:convergence-result}, let $0 = t_0 < t_1 < \dots < t_s=1$ be the partition whose existence is guaranteed by Theorem \ref{thm:convergence-result}. Then each of the $s$ applications of NM after the Euler prediction step in Algorithm \ref{alg:euclidean-distance-retraction} converges, and hence Algorithm \ref{alg:euclidean-distance-retraction} converges, and outputs the correct critical point $R_p(v)$.
    \begin{proof}
    If we Taylor expand $G(z_{u+\Delta u}, u+\Delta u)$ in $z, u$ and assume that $z_{u+\Delta u}$ is a zero of $G$ for the parameter $u+\Delta u$, we get the equation
\begin{eqnarray*}
0 = G(z_{u+\Delta u}, u+\Delta u) &=& \underbrace{G(z,u)}_{=0} + \frac{\partial G(z,u)}{\partial z}\cdot (z_{u+\Delta u}-z) +  \frac{\partial G(z,u)}{\partial u}\cdot \Delta u\\
&~&+ \frac{1}{2}\frac{\partial^2 G(\xi,\mu)}{\partial z^2}\left(z_{u+\Delta u} - z, z_{u+\Delta u} - z\right)
\end{eqnarray*}
for some $(\xi,\mu)$ on the line segment between $(z,u)$ and $(z_{u+\Delta u}, u + \Delta u)$. Since $G(x,\lambda;u) = (g(x), dg_x^T \lambda - u + x)$ then the second derivatives with respect to $u$ are all zero, explaining their absence in the formula above. Since $\partial G/\partial z$ is invertible in the point $z,u$ by assumption, we can insert Equation (\ref{eqn:deltazexpression}) into the above expression and multiply both sides by the inverse of $\frac{\partial G}{\partial z}(z,u)$ from the left to obtain
\begin{eqnarray*}
z_{u+\Delta _u}-z-\Delta z &=& - \frac{1}{2}\left(\frac{\partial G(z,u)}{\partial z}\right)^{-1}\frac{\partial^2 G(\xi,\mu)}{\partial z^2}\left(z_{u+\Delta u} - z, z_{u+\Delta u} - z\right).
\end{eqnarray*}
After taking norms and using the same upper bound $M$ as appeared in the proof of Theorem \ref{thm:convergence-result} we find
\begin{eqnarray*}
\left\lvert \left\lvert z_{u+\Delta _u}-z-\Delta z \right\lvert\right\lvert &\leq&
\frac{1}{2} \left\lvert \left\lvert \frac{\partial G(z,u)}{\partial z}^{-1}\right\lvert \right\lvert \cdot \left\lvert\left\lvert \frac{\partial^2 G(\xi,\mu)}{\partial z^2}\right\lvert \right\lvert \cdot \left\lvert\left\lvert z_{u+\Delta u} - z \right \lvert \right\lvert ^2 \\
&\leq & M ||z_{u+\Delta u}-z||^2.
\end{eqnarray*}
This implies that $z+\Delta z$ lies in the region of quadratic convergence of Newton's method, completing the proof.
    \end{proof}
\end{corollary}

Therefore, using Euler prediction to get the initialization $z+\Delta z$ for Newton's method is at least as good as using the original point $z$. 

\begin{remark}
Runge-Kutta schemes \cite[p.282]{numericalanalysis} are higher-order methods than Euler's method, so often times they improve the predictor's accuracy and consequently allow for a coarser discretization of the unit interval. Moreover, they provide a priori error bounds on the distance of the predictor to the solution path without knowing it. This data can be used to choose step sizes adaptively, ensuring that NM converges quickly at every step \cite{mixedprecisionpathtracking}. For these reasons, higher-order schemes are used in modern homotopy continuation algorithms, much rather than Euler's method \cite{efficientpathtracking, Bertini, HC.jl, nag-macaulay2}. Finally, there are methods to avoid jumping between solution paths, making the path-tracking algorithm even more robust \cite{robusthomotopycontinuation}.
\end{remark}

\color{red}

\color{black}

\section{Statistics: Maximum Likelihood Retraction}
\label{section:statistical-models}

In this section we leave the setting of submanifolds of Euclidean space and instead consider statistical models as submanifolds of a more general Riemannian manifold whose metric varies from point to point. In this context, we will show that if minimizing Euclidean distance is replaced by maximizing likelihood, then we can define analogous retraction maps to optimize functions on statistical models $\mathcal{M}$.
Theorem \ref{thm:fisher-retraction} shows how to make these retractions second-order accurate for the Levi-Civita connection on $\mathcal{M}$.

Consider the smooth manifold structure on $\R_{>0}^n$ induced by viewing the identity map as a global chart. We will use the coordinates denoted $(x_i)$ for this global chart for all of our computations. 
The chart-induced global frame for the cotangent bundle over $\R_{>0}^n$ is denoted $(dx_1,\dots,dx_n)$, while $(\partial_1,\dots,\partial_n)$ will denote the chart-induced global frame for the tangent bundle. Recall, a global frame of a smooth vector bundle is an ordered tuple of global sections whose values at every point yield a basis. The values at a point $x \in \R_{>0}^n$ will be denoted $\partial_i|_x$, so that the vectors $(\partial_1|_x,\dots,\partial_n|_x)$ form a basis for $T_x \R_{>0}^n$.

Let $\Delta_{n-1} \subset \R_{>0}^n \subset \R^n$ be the open probability simplex defined by
\begin{equation*}
    \Delta_{n-1} = \left\{ (x_1,\dots,x_n) \in \R_{>0}^n: \sum_{i=1}^n x_i = 1 \right\}.
\end{equation*}
Given $u \in \Delta_{n-1}$ we define the log-likelihood function $\ell_u:\R_{>0}^n \to \R$ by
\begin{equation*}
    (x_1,\dots,x_n) \mapsto \sum_{i=1}^n u_i \log(x_i) = \log \left( \prod_{i=1}^n x_i^{u_i}\right).
\end{equation*}
A smooth statistical model is a subset $\mathcal{M} \subset \Delta_{n-1}$ which is also a smooth manifold. In contrast, an algebraic statistical model is a subset $\mathcal{M} \subset \Delta_{n-1}$ which is also a real algebraic variety. Given a point $u \in \Delta_{n-1}$ coming from the results of an experiment, maximum likelihood estimation seeks to find the point $p \in \mathcal{M}$ which maximizes the likelihood of observing the results $u$. Such a point also maximizes $\ell_u$ restricted to $\mathcal{M}\subset \Delta_{n-1}$. For more details, see \cite[Ch. 6]{mathematical-statistics-text} or \cite{algebraicstatisticforcompbio}.

A Riemannian metric is a smooth symmetric covariant 2-tensor field that is positive definite at each point. In the global chart on $\R_{>0}^n$, the \textit{Fisher metric} is defined by
\begin{equation*}
    g = \sum_{i,j=1}^n g_{ij} \,\, dx_i \otimes dx_j
\end{equation*}
where $g_{ii} = \frac{1}{x_i}$ and $g_{ij} = 0$ for any $i \neq j$ are smooth functions written in the global chart coordinates. Indeed, it is a Riemannian metric (see \cite[p. 31]{informationgeometry}).
In statistics and information geometry, the Fisher metric arises naturally, being the unique metric invariant under sufficient statistics \cite[Thm 1.2]{informationgeometry}. In maximum likelihood estimation, one chooses the parameters of a model so as to maximize the likelihood of the observed data. Since the Fisher metric appears in the Cram\'er-Rao inequality \cite[Thm 1.3]{informationgeometry}, it is also related to the reliability of the estimator.

Given a function $f$ on a manifold, the vector field $\text{grad } f$ is given in local coordinates \cite[p. 27]{lee2018} by
\begin{equation*}
    \sum_{j=1}^n \left( \sum_{i=1}^n g^{ij} \partial_i f \right) \partial_j
\end{equation*}
where the $g^{ij}$ are smooth functions giving the inverse of the matrix $g_{ij}$ of the metric. For the Fisher metric, $g^{jj} = x_j$ and $g^{ij} = 0$ for $i \neq j$. Since $\partial_j \ell_u = u_j/x_j$ we see that
\begin{equation*}
    \text{grad } \ell_u = \sum_{j=1}^n \left( x_j \frac{u_j}{x_j} \right) \partial_j = \sum_{j=1}^n u_j \, \partial_j.
\end{equation*}
Compare this with the Euclidean gradient, which we denote $\nabla \ell_u$, given by
\begin{equation*}
    \nabla \ell_u = \sum_{j=1}^n \frac{u_j}{x_j} \, \partial_j
\end{equation*}
in the same global frame, but computed with the Euclidean metric $g_{ij} = \delta_{ij}$. The metrics themselves will be denoted by $\langle \cdot, \cdot \rangle^{(f)}_p$ and $\langle \cdot, \cdot \rangle^{(e)}_p$ when they act as inner products on the tangent space $T_p \R^n_{>0}$. Accordingly, $N_p^{(f)} \mathcal{M}$ and $N_p^{(e)} \mathcal{M}$ denote the normal spaces inside $T_p \R^n_{>0}$ with respect to to the Fisher and Euclidean metrics. Viewing $\mathcal{M}$ as a Riemannian submanifold of $\R^n_{>0}$, we will use the Levi-Civita connection (cf. \cite[Ch. 5]{lee2018}) corresponding to the Fisher metric in our computations of covariant derivatives below.

\begin{theorem}\label{thm:fisher-retraction}
Let $\mathcal{M} \subset \Delta_{n-1} \subset \R_{>0}^n$ be a smooth statistical model. For $p \in \mathcal{M}$ and $v = \sum v_i \, \partial_i|_p \in T_p \mathcal{M}$, consider the relation $R_p \subset T_p \mathcal{M} \times \mathcal{M}$ defined by $\{ (v,q^*) : v\in T_p\mathcal{M},~ q^* \in \text{argmax}_{q \in \mathcal{M}} \,\, \ell_{u(p,v)}(q) \}$ where $u(p,v) \in \R^n$ has components $u_i = p_i + v_i + \frac{v_i^2}{4p_i}$. Then there exists a neighborhood $U \subset T_p \mathcal{M}$ of $0$ where the argmax is unique and the relation $R_p|_U$ defines a function
\begin{align*}
    R_p|_U: U &\to \mathcal{M}\\
    v &\mapsto \text{argmax}_{q \in \mathcal{M}} \,\, \ell_{u(p,v)}(q),
\end{align*}
which is a first-order, local retraction on $\mathcal{M}$ at $p$. 

Furthermore, viewing $\mathcal{M}$ as a Riemannian submanifold of $\R_{>0}^n$ with the Fisher metric and its Levi-Civita connection, $R_p|_U$ is a second-order retraction, so that the curve $t \mapsto R_p(tv)$ matches geodesics on $\mathcal{M}$ up to second-order at $t=0$.
\end{theorem}

\begin{proof}
First we check that there exists a neighborhood $U$ of $0 \in T_p \mathcal{M}$ such that $R_p \cap (U \times \mathcal{M})$ actually defines a function $R_p:U \to \mathcal{M}$. Then we will check that this function is a retraction. Consider the normal bundle $N^{(f)}\mathcal{M}$ as a submanifold of $\R^n \times \R^n$ and define $E:N^{(f)}\mathcal{M} \to \R^n$ by $E(y,w) = (y_i + w_i)_{i=1}^n$ in the global coordinates, denoted $y+w$ for short. This is a smooth map, being the restriction of the addition map on $\R^n$. Its differential at $(p,0)$ is bijective. To see this, consider $E$ restricted to the subset $\mathcal{M}_0 = \{ (y,0) : y \in \mathcal{M} \}$. Then $E|_{\mathcal{M}_0}$ is a diffeomorphism of $\mathcal{M}_0$ onto $\mathcal{M}$ and hence its differential at $(p,0)$ is bijective to $T_p \mathcal{M}$. Restricting $E$ to $N_0 = \{(p,w) : w \in N_p^{(f)} \mathcal{M} \}$ gives a map onto the affine subspace $p + N_p^{(f)} \mathcal{M} \subset \R^n$ whose differential is also bijective. Since $\R^n = T_p \R_{>0}^n = T_p \mathcal{M} \oplus N_p^{(f)} \mathcal{M} $ we see that $dE_{(p,0)}$ is bijective. Therefore $E$ is a local diffeomorphism at $(p,0)$ by the inverse function theorem for manifolds \cite[Thm. 4.5]{Lee-manifolds}. In particular, there is a neighborhood $W$ of $(p,0)$ in $N^{(f)}\mathcal{M}$ where $E$ is injective. 

The image of this neighborhood under $E$ is an open neighborhood of $p$ in $\R^n$ for which the maximum likelihood problem has a unique solution. Say $w \in E(W)$ had two points $y, z \in \mathcal{M}$ which maximized the log-likelihood function $\ell_w$ restricted to $\mathcal{M}$. By the usual first-order criteria for local maxima constrained to a subset of $\R^n$, we know that 
the Euclidean gradients $\nabla \ell_w|_y$ and $\nabla \ell_w|_z$ 
lie in the Euclidean normal spaces $N_y^{(e)}\mathcal{M}$ and $N_z^{(e)}\mathcal{M}$, respectively.
We claim that this forces both $\sum_i (w_i - y_i)\partial_i|y$ and $\sum_i(w_i - z_i)\partial_i|z$ to be in  $ N_y^{(f)} \mathcal{M}$ and $ N_z^{(f)} \mathcal{M}$, respectively. To see this we calculate with an arbitary tangent vector $\nu \in T_y \mathcal{M}$,
\begin{align*}
    \langle \nu, w-y \rangle^{(f)}_y &= \langle \nu, w \rangle^{(f)}_y - \langle \nu, y \rangle^{(f)}_y\\
    &= \langle \nu, \nabla \ell_w|_y \rangle^{(e)}_y - \langle \nu, \mathbf{1} \rangle^{(e)}_y.
\end{align*}
The first term is zero by first-order criteria for local maxima, and the second term is zero because $\mathcal{M} \subset \Delta_{n-1}$ is a statistical model, and hence lives in an affine hyperplane whose Euclidean normal vector is the all ones vector, since probabilities sum to one. 
In this case, the symbol $\mathbf{1}$ denotes the all ones vector $\sum_j \partial_j|_y$. 
Therefore $\sum_i (w_i - y_i)\partial_i|y \in N_y^{(f)} \mathcal{M}$. Similarly, $ \sum_i (w_i - z_i) \partial_i|_z \in N_z^{(f)} \mathcal{M}$.
But then $E(y, w-y) = y + w - y = w$ and $E(z,w-z) = z + w - z = w$ contradicts the injectivity of $E$ on $W$. Therefore $E(W)$ is a neighborhood of $p \in \R^n$ for which the maximum likelihood problem has a unique solution. 
Let $u(p,\cdot):T_p \mathcal{M} \to \R^n$ be the map sending a tangent vector $v = \sum v_i \partial_i|_p$ to the tuple $\left(p_i + v_i + \frac{v_i^2}{4p_i} \right)_{i=1}^n$, where all coordinates are in terms of the global chart. Then there exists a neighborhood $U$ of $0 \in T_p \mathcal{M}$ where $u(p,v) \in E(W)$ whenever $v \in U$, and hence the relation $R_p$ defines a function on $U$.

As a result of the previous discussion, for $p \in \mathcal{M}$ and $v \in U \subset T_p \mathcal{M}$ we can define the function
\begin{equation*}
    R_p(v) = \text{argmax}_{q \in \mathcal{M}} \, \ell_{u(p,v)}(q).
\end{equation*}
First we check that $R_p$ is smooth. Let $\pi:N^{(f)}\mathcal{M} \to \mathcal{M}$ be the projection onto $\mathcal{M}$, sending $(y,n) \mapsto y$.
For $v \in U \subset T_p \mathcal{M}$ we have $u(p,v) \in E(W)$. Since $E$ is a diffeomorphism between $W$ and $E(W)$ we know the inverse $E^{-1}$ exists and is smooth on $E(W)$. Since $R_p = \pi \circ E^{-1} \circ u(p,\cdot)$ is the composition of smooth maps, $R_p$ is also smooth.

Now we check that $R_p(0)=p$. But this follows since the argmax of the function $\ell_u$ for any $u \in \Delta_{n-1}$ over the larger domain $\Delta_{n-1}$ is unique and equal to $u$. This fact is well-known, and can be checked by noting that $u$ is the only critical point of $\ell_u$ on $\Delta_{n-1}$ and that the Hessian is negative definite since $\Delta_{n-1} \subset \R^n_{>0}$.
Therefore since $\mathcal{M} \subset \Delta_{n-1}$, when $v=0$ we have a unique argmax at $u(0) = p+0=p \in \mathcal{M} \subset \Delta_{n-1}$ as needed.

Finally, we check the conditions for a first-order and second-order retraction. We don't need the Riemannian structure for first-order, but since we are proving second-order as well we treat both cases uniformly. Define $c(t) = R_p(tv)$ for $t$ in a small enough neighborhood $\mathcal{I} \subset \R$ of zero, such that $\mathcal{I}v \subset U$. Then $c(t)$ is the unique point in $\mathcal{M}$ which maximizes the log-likelihood function $\ell_{u(p,tv)}$ restricted to $\mathcal{M}$. Let $\Dot{c}$ be the vector field on the curve $c$ given in the global chart by $\Dot{c} = \sum_{j=1}^n \Dot{c_j}(t) \, \partial_j$, where $\Dot{c_j}(t)$ denotes the derivative of the component function $c_j(t)$ of the curve $c:\mathcal{I} \to \mathcal{M}$. To show that $R_p$ is a first-order retraction we need to show that $v = \sum_j v_j \partial_j|_p$ is equal to $\Dot{c}|_p = \sum_j \Dot{c_j}(0) \partial_j|_p$. To show that $R_p$ is a second-order retraction we need to show that $D_t \Dot{c} |_p \in N^{(f)}_p \mathcal{M}$ \cite[Def. 8.64 and Eqn 8.27]{boumal2020intromanifolds}. Here, $D_t$ denotes the covariant derivative along the curve $c$, using the Levi-Civita connection for the Fisher metric on $\mathcal{M}$ (See \cite[Thm 4.24]{lee2018}). 

Let $A$ be any vector field on the curve $c$ which is tangent to $\mathcal{M}$, so that $A(t) = A|_{c(t)} \in T_{c(t)}\mathcal{M}$ for all $t \in \mathcal{I}$. By first-order criteria for local maxima we know that $\langle A(t), \nabla \ell_{u(t)} \rangle^{(e)}_{c(t)} = 0$ for all $t \in \mathcal{I}$, since $A$ is everywhere tangent to $\mathcal{M}$. We also know that $\langle A(t), \mathbf{1} \rangle^{(e)}_{c(t)} = 0$ since $\mathcal{M} \subset \Delta_{n-1}$.
By $u(t) = u(p,tv)$ denote the curve in $E(W)$ with components $p_i + tv_i + t^2\frac{v_i^2}{4p_i}$, and let $B = B(t)$ be the vector field along the curve $c$ given by $\sum_j B_j(t) \partial_j|_{c(t)}$ where $B_j(t) = u_j(t) - c_j(t) =  p_j + t v_j + t^2\frac{(v_j)^2}{4p_j} - c_j(t)$.
Then, we have
\begin{align*}
    0 &= \langle A(t), \nabla \ell_{u(t)} \rangle^{(e)}_{c(t)} - \langle A(t), \mathbf{1} \rangle^{(e)}_{c(t)}\\
    &= \langle A(t), \nabla \ell_{u(t)} - \mathbf{1} \rangle^{(e)}_{c(t)}\\
    &= \langle A(t), \sum_{i=1}^n (u_i(t) - c_i(t))\partial_i|_{c(t)} \rangle^{(f)}_{c(t)}\\
    &= \langle A(t), B(t) \rangle^{(f)}_{c(t)}.
\end{align*}
Note that $B(t)$ is the orthogonal projection of $\text{grad } \ell_{u(p,tv)}$ onto $T_{c(t)} \Delta_{n-1}$ along the curve $c$ when using the Fisher metric. 
Applying $\frac{d}{dt}$ and $\frac{d^2}{dt^2}$ using \cite[Prop 5.5]{lee2018} we find
\begin{align}
    0 &= \frac{d}{dt} \langle A, B \rangle^{(f)}_{c(t)} = \langle D_t A, B \rangle^{(f)}_{c(t)} + \langle A, D_t B \rangle^{(f)}_{c(t)} \label{eqn:one-fish} \\
    0 &= \frac{d^2}{dt^2} \langle A, B \rangle^{(f)}_{c(t)} = \langle D_t D_t A, B \rangle^{(f)}_{c(t)} + 2 \langle D_t A, D_t B \rangle^{(f)}_{c(t)} + \langle A, D_t D_t B \rangle^{(f)}_{c(t)}. \label{eqn:two-fish}
\end{align}
Recall (see \cite[p. 102]{lee2018}) the covariant derivative of a vector field along a curve $c$ is given by
\begin{equation}\label{eqn:covariant-derivative-B}
    D_t B(t) = \sum_k \left( \Dot{B}_k(t) + \sum_i \sum_j \Dot{c}_i(t) B_j(t) \Gamma^k_{ij}(c(t)) \right) \partial_k|_{c(t)}
\end{equation}
where $\Gamma^k_{ij}$ are the Christoffel symbols of the metric in the given local frame. For the Fisher metric with the corresponding Levi-Civita connection in the global frame $\partial_k$ induced by the global chart, every $\Gamma^k_{ij}=0$ except $\Gamma^j_{jj}(c(t)) = \frac{-1}{2c_j(t)}$, evaluating at the point $c(t)$ along the curve (cf. \cite[p. 123]{lee2018}). Using this in our calculations, equation (\ref{eqn:covariant-derivative-B}) becomes $D_t B(t) = \sum_k \left( \Dot{B}_k(t) - \frac{\Dot{c}_k(t) B_k(t)}{2c_k(t)} \right) \partial_k|_{c(t)}$, but at $t=0$ this becomes $D_t B(0) = \sum_k \left( v_k - \Dot{c}_k(0) \right) \partial_k|_p$. Then, since $c(0)=p$ we have $B(0) = 0$, and since both $v$ and $\Dot{c}|_p$ are in the tangent space at $p$, equation (\ref{eqn:one-fish}) implies that $v_k = \Dot{c}_k(0)$ and so $v = \Dot{c}|_p$, which completes the proof that $R_p$ is a first-order retraction.

Having proved first-order, both $B(0)$ and $D_t B(0)$ are zero, and the first two terms of equation (\ref{eqn:two-fish}) vanish at $t=0$. A somewhat tedious calculation (which we omit) shows that $D_t D_t B(0) = \sum_k \left( \frac{(v_k)^2}{2p_k} - \Ddot{c}_k(0) \right) \partial_k|_p$. Calculating $D_t \Dot{c}(0)$ using equation \ref{eqn:covariant-derivative-B}, it is noted that $D_t D_t B(0) = - D_t \Dot{c}(0)$. Therefore, equation (\ref{eqn:two-fish}) becomes
\begin{equation*}
    \langle A(0), \, - D_t \Dot{c}(0) \rangle^{(f)}_p = 0.
\end{equation*}
Since this is true for any vector field $A$ which is everywhere tangent to $\mathcal{M}$, and hence for any tangent vector $A(0) \in T_p \mathcal{M}$, we have $D_t \Dot{c}(0) \in N_p^{(f)} \mathcal{M}$, which completes the proof.
\end{proof}

Computing the maximum likelihood retraction can also be done using homotopy continuation on an appropriate Lagrange multiplier system of equations using the log-likelihood function.

\section{Description of the Software}\label{section:usage}

The software package HomotopyOpt.jl has been written in Julia, a high-level, dynamic programming language. In particular, this gives us access to HomotopyContinuation.jl \cite{HC.jl}, a numerical algebraic geometry package. Note that because we use this package, the map $g:\R^n \to \R^m$ defining $\mathcal{M}$ must have polynomial component functions, while the objective function can be anything smooth. Assume now that we want to optimize the smooth function $f(x_1,x_2,x_3) = 2^{({x_2}-1)^2}$ subject to the constraint set 
\[g = \begin{bmatrix} {x_1}^2+{x_2}^2+{x_3}^2-1 \\ x_3-{x_1}^3\end{bmatrix}.\]
To solve this optimization problem with initial point $(0,-1,0)\in \mathcal{M}$ we write:

\begin{lstlisting}
using HomotopyOpt # load package

p = [0,-1.,0]
f = x -> 2^((x[2]-1)^2)
g = x -> [x[1]^2+x[2]^2+x[3]^2-1, x[3]-x[1]^3]
M = ConstraintVariety(g, 3, 1) # M embedded in R^3 with dim 1
res = findminima(p, 1e-5, M, f) # Set the tolerance to 10^(-5)
\end{lstlisting}

Now, we can use Lagrange multipliers to calculate analytically that the minimum of $f$ on $\mathcal{V}(g)$ is at $(0,1,0)$. Indeed, by typing \texttt{res.computedpoints[end]} in the command prompt the final point can be examined, revealing the vector 
\[\texttt{ [-0.0012328615227131664 0.9999992400259441 -1.8738848314414474e-9], }\] which is a solid approximation of the true optimum.

As described in Section \ref{section:euclidean-distance-retraction}, to solve this optimization problem the software takes several local steps (\texttt{maxlocalsteps}), 
running \textit{Riemannian gradient descent} \cite[p. 62]{boumal2020intromanifolds}.
First, it calculates a descent direction $v$ in $T_{x}V$, given by the negative Riemannian gradient obtained by projecting $f$'s gradient onto $T_xV$. Afterwards, the algorithm uses backtracking line search with Wolfe conditions to compute a feasible step size corresponding to the previously chosen descent direction (see \cite{loclipschitzlinesearch, Ring2012}). These conditions provide a lower and an upper bound on admissible step sizes and come alongside convenient theoretical guarantees, such as existance of open intervals of step sizes satisfying the Wolfe conditions \cite{loclipschitzlinesearch}. Furthermore, provided a sequence of step sizes satisfying the Wolfe conditions and well-chosen descent directions such as steepest descent, Zoutendijk's Theorem (cf. \cite{Ring2012}) guarantees convergence for the corresponding Riemannian optimization algorithm.
After having chosen a step size $\alpha$ and a descent direction $v$, the linear steps $p+\alpha v$ are corrected to the variety using the Euclidean Distance Retraction Algorithm \ref{alg:euclidean-distance-retraction}. 

Having completed \texttt{maxlocalsteps} local steps, we check if the first-order criterion for optimality $\text{grad}(x_k)\approx 0$ is satisfied, and also whether the algorithm is stuck in a saddle point or singularity. Our method recognizes saddle points by looking for insufficient decrease in the objective function and singularities by comparing the current Jacobian's (numerical) rank to the previous Jacobian's rank. If it changed, we need to assume that the iteration is currently in or close to a singularity. In either of these cases, we proceed to check whether the current iterate is a minimum by employing the second-order criterion on optimality (see \cite[Prop. 6.3]{boumal2020intromanifolds}). If it is inconclusive or finds that the current point is not optimal, the algorithm continues. In the case of a singularity, additionally we try to resolve the singularity by intersecting the variety with a sphere centered at the current iterate (and possibly additional hyperplanes) and choosing the intersection point with the smallest objective function value to continue the optimization. Otherwise, if the second-order criterion has a positive result, the current point is returned to the user as optimal. Finally, if the algorithm was inconclusive and exceeds the prescribed maximal amount of time \texttt{maxseconds}, the current point is returned with the additional information \texttt{lastpointisminimum=false}.

There are a few additional ways to interact with this package that deserve mentioning. The method \texttt{watch} animates the optimization procedure's steps on the manifold, \texttt{draw} produces a static picture of both the local steps and a global picture and \texttt{setequationsatp!} randomizes the equations at a specific point, so there are only codimension-many left. For this, the dimension is calculated at the specified point by checking the constraint system's Jacobian rank. 

\section{Experiments}
\label{section:experiments}
As discussed in Section \ref{section:homotopy-continuation}, the Euclidean distance retraction in combination with homotopy continuation is well-suited for solving Riemannian optimization problems. Despite having favorable theoretical properties, it remains unclear, why many small Euler steps should be more efficient relative to directly solving the homotopy system $G(x,\lambda; p+v)=0$, which we previously called \textit{Algorithm 0}. This section therefore aims at comparing the performance of using homotopy continuation in Algorithm \ref{alg:euclidean-distance-retraction} with Newton's method (e.g. Algorithm 0, see \cite[Ex. 4]{Adler-newtonsmethod-retraction-definition}).

In the following, three smooth manifolds are considered. To get a reasonable global picture of the algorithms, each of them is intersected with $100$ random hyperplanes to obtain starting points. Both algorithms are then run with each of these samples as initialization, recording the following benchmarks:
\begin{itemize}
    \item Percentage of optimization paths that converge in $10$ seconds.
    \item The average amount of local steps the algorithm takes per starting point. Each execution of the backtracking line search algorithm is counted exactly once.
    \item The average amount of time each method takes per starting point.
\end{itemize}
\noindent
The test instances are available in the \texttt{HomotopyOpt.jl}\footnote{\url{https://github.com/matthiashimmelmann/HomotopyOpt.jl}} repository and were run on a Windows 11 64-Bit machine with an Intel i7-10750H @ 2.6GHz processor and 16GB RAM. The figures in this section are produced with a \texttt{GLMakie} backend \cite{GLMakie} inside the package \texttt{Implicit3DPlotting.jl}\footnote{\url{https://github.com/matthiashimmelmann/Implicit3DPlotting.jl}} designed to visualize implicit space curves and surfaces. 

As a first instance to test our algorithm on, consider the sextic planar curve
\[\mathcal{V}_1 = \{(x,y)\in \mathbb{R}^2 ~:~(x^4+y^4-1)\cdot (x^2+y^2-2)+x^5 y=0\}.\]
The Euclidean distance to the point $(2,0.5)$ is chosen as objective function. The four minima of this optimization problem are displayed in Figure \ref{fig:sextic}.
\begin{figure}
\begin{minipage}[t]{0.478\linewidth}
\centering
\includegraphics[width=0.88\linewidth]{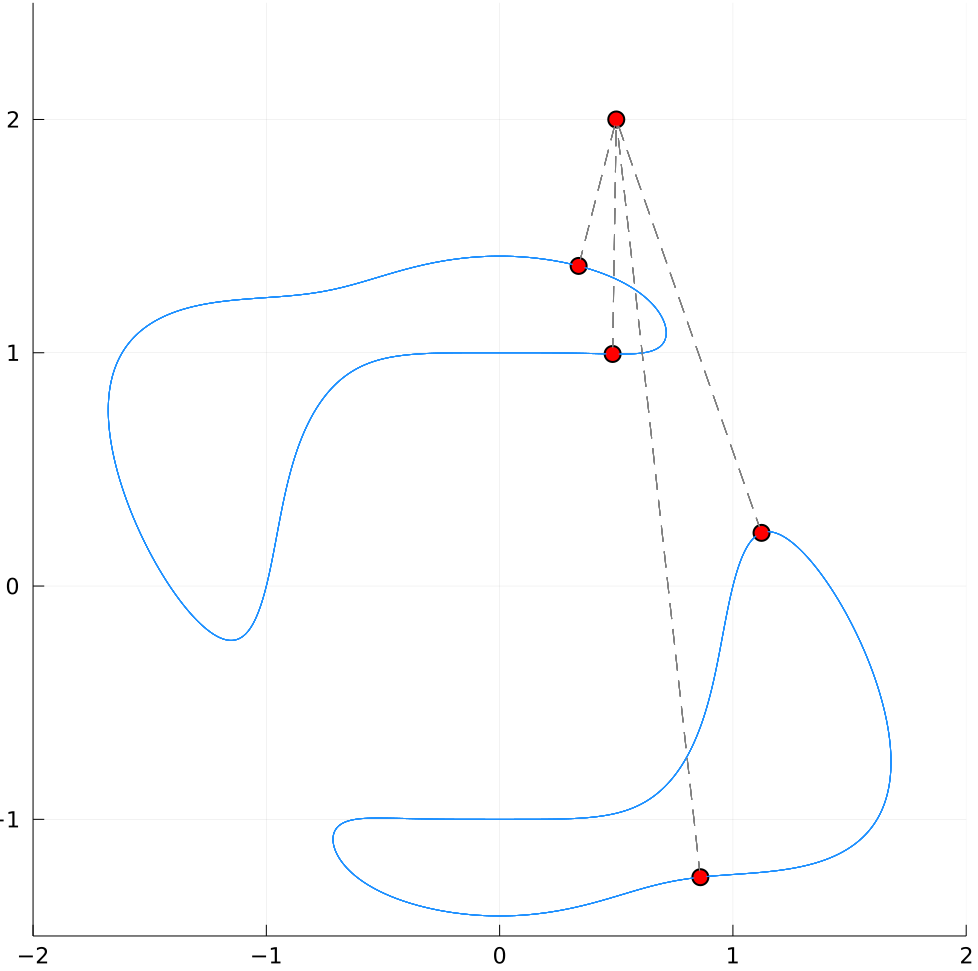}
\label{fig:sextic}
\caption{\small Minima on $\mathcal{V}_1$ w.r.t. the objective function $f_1(x,y)=(x-2)^2+(y-0.5)^2$.}
\end{minipage}
\hspace*{5mm}
\begin{minipage}[t]{0.478\linewidth}
\centering
\includegraphics[width=0.984\linewidth]{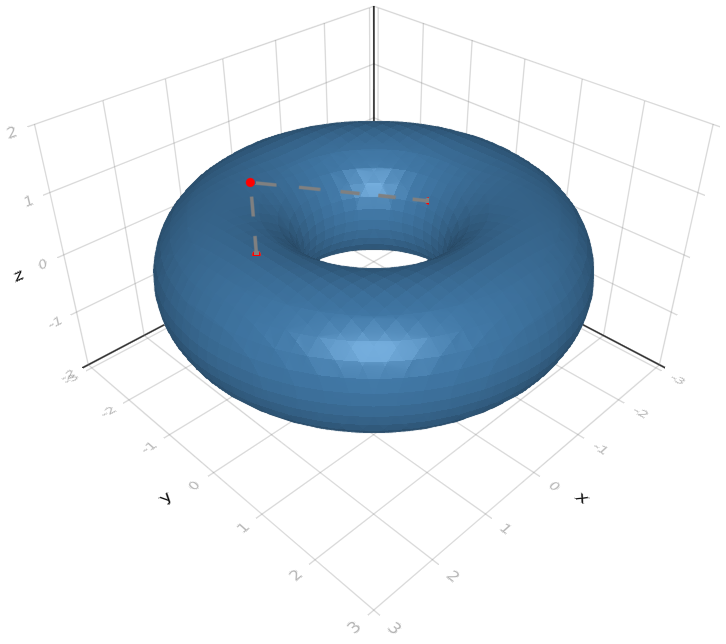}
\label{fig:torus}
\caption{\small Minima on $\mathcal{V}_2(2,1)$ w.r.t. the objective function $f(x,y,z)=(x-2)^2+y^2+(z-2)^2$.}
\end{minipage}
\end{figure}

Table \ref{tab:V1} suggests that homotopy continuation and Newton's method perform comparably with respect to amount of converged paths and average amount of local steps. However, homotopy continuation performs five times better when it comes to average runtime. As discussed in Section \ref{section:euclidean-distance-retraction}, one reason could be slow convergence when the initial point is outside the region of quadratic convergence.
\begin{table}[h]
\centering
\begin{tabular}{ c c c c }
~ & converged & $\varnothing$ steps & $\varnothing$ time \\ 
Homotopy Continuation & $100\%$ & 9.8 & 9.9 ms\\  
Newton's Method & $100\%$ & 8.8 & 42.9 ms
\end{tabular}
\caption{\small The results that $V_1$ produced.}
\label{tab:V1}
\end{table}

For our second example, consider the torus with distance from its center to the tube's central axis $R_1$ and the tube's radius $R_2$:
\[\mathcal{V}_2(R_1,R_2) = \{(x,y,z)\in \mathbb{R}^3 ~:~(x^2+y^2+z^2-{R_1}^2-{R_2}^2)^2+4{R_1}^2z^2-4{R_1}^2{R_2}^2=0\}.\]
When $R_1>R_2>0$, $\mathcal{V}_2(R_1,R_2)$ is a smooth, 2-dimensional manifold. Therefore, we choose $R_1=2$ and $R_2=1$. As objective function, we consider the Euclidean distance to the point $(2,0,2)$. Figure \ref{fig:torus} then depicts the two minima of the corresponding constraint optimization problem.

Again, Table \ref{tab:V2} demonstrates that homotopy continuation and Newton's method perform similarly in the first two categories, while the average runtime of Newton's method is over $3$ times higher than that of homotopy continuation.

\begin{table}[h!]
\centering
\begin{tabular}{ c c c c }
~ & converged & $\varnothing$ steps & $\varnothing$ time \\ 
Homotopy Continuation & $100\%$ & 10.9 & 14.6 ms\\  
Newton's Method & $100\%$ & 10.9 & 44.8 ms
\end{tabular}
\caption{\small The results that $V_2(2,1)$ produced.}
\label{tab:V2}
\end{table}

As a final example, consider the twisted cubic
$\mathcal{V}_3 = \{(x,y,z) \in \mathbb{R}^3 ~:~x^2=y~\wedge ~ x^3=z\};$
a space curve with parametrization $t\mapsto (t,t^2,t^3)$. This time, our objective function of choice is the Euclidean distance from the point $(0,-1,1)$. Figure \ref{fig:twistedcubic} depicts the minimum $(0,0,0)$ of the corresponding constraint optimization problem.

\begin{figure}[h!]
\centering
\includegraphics[width=0.54\linewidth]{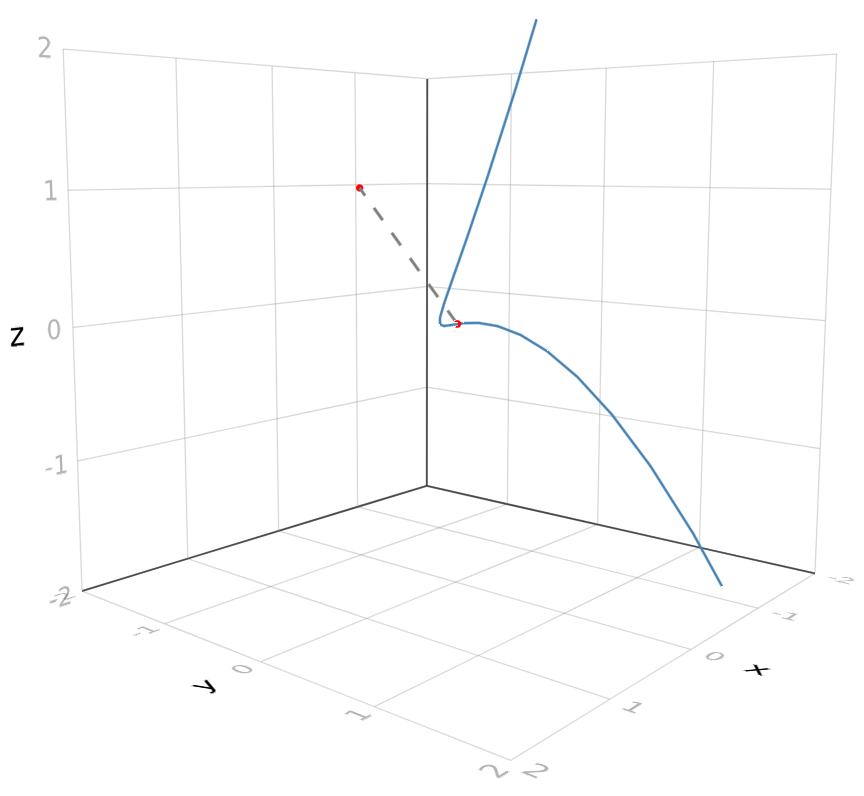}
\label{fig:twistedcubic}
\caption{\small Minima on $\mathcal{V}_3$ w.r.t. the objective function $f(x,y,z)=x^2+(y+1)^2+(z-1)^2$.}
\end{figure}

Similar to the other two examples, Table $\ref{tab:V3}$ illustrates that homotopy continuation again performs significantly better when it comes to average runtime.

\begin{table}[h!]
\centering
\begin{tabular}{ c c c c }
~ & converged & $\varnothing$ steps & $\varnothing$ time \\ 
Homotopy Continuation & $100\%$ & 8.8 & 10.4 ms\\  
Newton's Method & $100\%$ & 9.0 & 27.2 ms
\end{tabular}
\caption{\small The results that $V_3$ produced.}
\label{tab:V3}
\end{table}

To conclude, the theoretical observations about Newton's method from Section \ref{section:euclidean-distance-retraction} were reproducible in experiments. In all instances, Newton's method was significantly slower than the homotopy continuation Algorithm \ref{alg:euclidean-distance-retraction}.

 {\footnotesize\linespread{0.8}
\bibliographystyle{abbrv}
\bibliography{references}

Alexander Heaton, Lawrence University, Appleton, WI, USA

Matthias Himmelmann, University of Potsdam, Potsdam, Germany 
}

\end{document}